\documentclass[10pt]{amsart}
\usepackage[margin=1.2in,marginparsep=0.1in,marginparwidth=1in]{geometry}
\usepackage{amssymb,amsmath,amsthm,amstext,amscd,latexsym,graphics,graphicx,bbm,caption}
\usepackage[usenames,dvipsnames,svgnames,table]{xcolor}
\usepackage{blindtext}
\usepackage[plainpages=false,colorlinks=true, pagebackref]{hyperref}
\usepackage{tikz}
\usepackage[all]{xy}
\usetikzlibrary{positioning}

\tikzset{>=stealth}
\hypersetup{citecolor=Sepia,linkcolor=blue, urlcolor=blue}

\usepackage{cite}   

\makeatletter
\def\@tocline#1#2#3#4#5#6#7{\relax
  \ifnum #1>\c@tocdepth 
  \else
    \par \addpenalty\@secpenalty\addvspace{#2}%
    \begingroup \hyphenpenalty\@M
    \@ifempty{#4}{%
      \@tempdima\csname r@tocindent\number#1\endcsname\relax
    }{%
      \@tempdima#4\relax
    }%
    \parindent\z@ \leftskip#3\relax \advance\leftskip\@tempdima\relax
    \rightskip\@pnumwidth plus4em \parfillskip-\@pnumwidth
    #5\leavevmode\hskip-\@tempdima
      \ifcase #1
       \or\or \hskip 2em \or \hskip 2em \else \hskip 3em \fi%
      #6\nobreak\relax
    \dotfill\hbox to\@pnumwidth{\@tocpagenum{#7}}\par
    \nobreak
    \endgroup
  \fi}
\makeatother

\newcommand{\A}{\mathbb{A}}

\newcommand{\C}{\mathbb{C}}

\newcommand{\F}{\mathbb{F}}
\newcommand{\Z}{\mathbb{Z}}

\newcommand{\Q}{\mathbb{Q}}
\newcommand{\R}{\mathbb{R}}

\newcommand{\p}{\mathfrak{p}}
\newcommand{\q}{\mathfrak{q}}

\newcommand{\Aa}{\mathfrak{a}}

\newcommand{\OK}{\mathcal{O}}

\DeclareMathOperator{\ab}{ab}

\DeclareMathOperator{\Gal}{Gal}

\DeclareMathOperator{\GL}{GL}

\DeclareMathOperator{\rk}{rk}
\DeclareMathOperator{\Sl}{\mathfrak{sl}}
\DeclareMathOperator{\SL}{SL}
\DeclareMathOperator{\tr}{tr}
\theoremstyle{plain}
\input xy
\xyoption{all}

\newtheorem{theorem}{Theorem}[section]
\newtheorem{corollary}[theorem]{Corollary}
\newtheorem{proposition}[theorem]{Proposition}
\newtheorem{lemma}[theorem]{Lemma}
\newtheorem{conjecture}[theorem]{Conjecture}
\theoremstyle{definition}
\newtheorem{definition}[theorem]{Definition}

\newtheorem{question}[theorem]{Question}
\theoremstyle{remark}
\newtheorem{remark}[theorem]{Remark}
\numberwithin{equation}{section}

\begin{document}

\author[R. Abdellatif]{Ramla Abdellatif} \address{Laboratoire Ami\'enois de Math\'ematique Fondamentale et Appliqu\'ee (LAMFA) -- Universit\'e de Picardie Jules Verne, 33 rue Saint-Leu , 80 039 Amiens Cedex 1, France} \email{ramla.abdellatif@math.cnrs.fr} 
\author[S. Pisolkar]{Supriya Pisolkar} \address{Indian Institute of
Science, Education and Research (IISER),  Homi Bhabha Road, Pashan,
Pune - 411008, India} \email{supriya@iiserpune.ac.in} 
\author[M. Rougnant]{Marine Rougnant} \address{Laboratoire de Math\'ematiques de Besan\c{c}on, UFR Sciences et techniques, 16, route de Gray, 25030 Besan\c{c}on cedex, France} \email{marine.rougnant@univ-fcomte.fr}
\author[L. Thomas]{Lara Thomas} \address{Institut Camille Jordan - Lyon/Saint-Etienne -- Universit\'e Lyon 1, 43, boulevard du 11 Novembre 1918, 69622 Villeurbanne, France} \email{lthomas@math.cnrs.fr}

\title{From Fontaine-Mazur conjecture to analytic pro-$p$-groups : a survey} 
\date{November 3rd, 2021 -- V3}

\begin{abstract} Fontaine-Mazur Conjecture is one of the core statements in modern arithmetic geometry. Several formulations were given since its original statement in 1993, and various angles have been adopted by numerous authors to try to tackle it. Among those, a range of tools that is not so well-known among young arithmetic geometers, goes back to Boston's seminal paper in 1992, and relies on purely group-theoretic methods (rather than representation-theoretic ones) to prove some special cases of this conjecture. Such methods have been later successfully carried on by Maire and his co-authors, and brings different informations on the objects involved in the conjecture. This survey article aims to review what is known in this direction and to present some interesting related questions the authors (aim to) work on.
\end{abstract}

\maketitle

\tableofcontents


\section{Introduction}\label{intro}
Fontaine-Mazur Conjecture is one of the core statements in modern arithmetic geometry. Several formulations were given since its original statement (as appeared in \cite{FM}), and various angles has been adopted by numerous authors to try to tackle it. To state the original Fontaine-Mazur Conjecture (FMC), we need to introduce some definitions and notations. Let $K$ be a number field, let $\overline{K}$ be a fixed separable closure of $K$ and let $G_{K}:=\Gal(\overline{K}/K)$ be the corresponding absolute Galois group. Given a prime number $p$, a finite extension $F$ of the field $\Q_{p}$ of $p$-adic numbers and a profinite group $G$, an {\it $F$-representation of $G$} is a finite-dimensional vector space over $F$ equipped with a continuous and linear action of $G$. When $F = \Q_{p}$, we call it a {\it $p$-adic representation of $G$}. A $p$-adic representation $\rho$ of $G_{K}$ is called {\it geometric} if it ramifies only at a finite number of places of $K$, and if for each place $v$ of $K$ above $p$, the restriction of $\rho$ to $G_{v}$ is potentially semi-stable in the sense of Fontaine \cite[Section 1.8]{F}. Finally, for any integer $r$, we let $\Q_{p}(r)$ denote the $r^{th}$ Tate twist of $\Q_{p}$, as defined in \cite[\S 3]{T}.

\begin{conjecture}[Fontaine, Mazur]\label{FMC1}
An irreducible $p$-adic representation of $G_{K}$ is geometric if, and only if, it is isomorphic to a subquotient of an \'etale cohomology group with coefficients in $\Q_{p}(r)$, for some $r \in \Z$, of a (projective, smooth) algebraic variety over $K$.
\end{conjecture}
\noindent In short, Fontaine-Mazur conjecture predicts that $p$-adic representations of global Galois groups that are potentially semi-stable at primes dividing $p$ and unramified outside finitely many places ought to come from algebraic geometry. \\

In the recent years, substantial progress has been made using $p$-adic representations and deformation theory, allowing for instance Kisin \cite{K} to prove the original conjecture for $2$-dimensional representations. Applied to $K = \Q$, this special case of the conjecture asserts that potentially semi-stable representations with odd determinant come from modular forms. 
\begin{conjecture}[Fontaine-Mazur Conjecture for $n = 2$]\label{FMC2} Let $\rho : \Gal(\overline{\Q}/\Q) \to GL_{2}(\Q_{p})$ be an odd, irreducible representation that is unramified outside finitely many primes and whose restriction to the decomposition group at $p$ is potentially semistable with distinct Hodge-Tate weights. Then $\rho$ is the twist of a Galois representation associated to a modular form of weight $k \geq 2$.
\end{conjecture}

\noindent Kisin's proof relies on an intimate connection between modularity lifting theorems, the Breuil-M\'ezard conjecture \cite{BM}, and Breuil's $p$-adic local Langlands correspondence for $\GL_{2}$ \cite{Br1, Br2}. We are here interested in a different approach, which can be motivated as follows.  First note that, together with a well-known conjecture of Tate, Conjecture \ref{FMC1} implies the following conjecture, which is elementary to state but completely out of reach for now.
 
\begin{conjecture}[Weak Fontaine-Mazur Conjecture]\label{FMC3} Every unramified pro-$p$-extension of $K$ whose Galois group is $p$-adic analytic is finite.
\end{conjecture} 
\noindent Now recall that in \cite{GS}, Golod and Shafarevich proved the existence of a number field $K$ and a prime $p$ such that $K$ admits an everywhere unramified infinite pro-$p$ extension $L$. Conjecture \ref{FMC3} then claims that the Galois group of this extension $L/K$ cannot be an infinite analytic pro-$p$ group (i.e. something that is isomorphic to a closed subgroup of $GL_{n}(\Z_{p})$ for some positive integer $n$). Hence the idea here is that a counter-example to Conjecture \ref{FMC3} would produce an everywhere unramified Galois representation with infinite image, which cannot ``come from algebraic geometry" in the sense of Conjecture \ref{FMC1}.\\
  
\noindent From Lazard's seminal work on $p$-adic analytic groups \cite[III, 3.4.3]{L}, we know that any finite-dimen- sional $p$-adic analytic group contains a finite index open uniform subgroup. Thus we can reformulate Conjecture \ref{FMC3} as follows. 

\begin{conjecture}[Uniform Fontaine-Mazur Conjecture (UFMC)]\label{FMCU} There is no number field $K$ for which there exists an infinite everywhere unramified Galois pro-$p$ extension $L$ such that $\Gal(L/K)$ is uniform. 
\end{conjecture} 
\noindent A major advantage in considering uniform groups is that they have a simple characterisation (in terms of filtration by their subgroups) that analytic pro-$p$-groups do not satisfy in general. (Recall that a finitely generated pro-$p$-group is analytic if, and only if, it has a powerful subgroup of finite index.) \\

The first evidence for the truth of Conjecture \ref{FMCU} were given by N. Boston in \cite{B1, B2}, using purely group-theoretic methods based on the connections between powerful pro-$p$-groups and uniform pro-$p$-groups instead of representation-theoretic tools. The goal of this paper is to introduce these purely group-theoretic tools, which are not so well-known among (young) arithmetic geometers, and to review some of the main results they bring about various conjectures related to Conjectures \ref{FMC1} and \ref{FMCU}. We also present some related questions on which is based current and future works of the authors. We think that this survey paper may be of interest to anyone willing to have a different viewpoint on Fontaine-Mazur Conjecture, which does not require advanced knowledge in $p$-adic representation theory and underlines the number-theoretic nature of the problem.\\

This paper is organised as follows. After gathering in Section \ref{uniform groups} the definitions and basic results we need about (uniform) pro-$p$-groups, we devote Section \ref{BostonProof} of this paper to study the proofs of the main result of \cite{B1}, which is the following special case of Conjecture \ref{FMCU}, and its generalisation to cyclic extensions of degree prime to $p$, which is the main result of \cite{B2}.

\begin{theorem}[Boston]\label{Boston-main} Given a prime number $p$ and a number field $F$, let $K$ be a normal extension of $F$ of prime degree $\ell \neq p$ and such that $p$ does not divide $h(F)$, the class number of $F$. Then there is no infinite, everywhere unramified, Galois, pro-$p$ extension $L$ of $K$ such that $L/F$ is Galois and $\Gal(L/K)$ is uniform. 
\end{theorem}

\noindent The main ingredient of the group-theoretic methods used to prove these results is the cyclic action on $\Gal(L/K)$ of a generator $\sigma$ of the Galois group $\Gal(K/F)$. This works quite well when the action of $\sigma$ on $\Gal(L/K)$ has no non-trivial fixed point. The study of such actions is in the spirit of what did later Hajir and Maire in \cite{HM}, where they attempted to extend Boston's strategy to the case of (tamely) ramified extensions $L/K$. The challenge here is to handle the fixed points introduced by ramification and, as a consequence, the constraints posed on the arithmetic of $L/K$.

This work of Hajir and Maire motivates our interest in the two following Fontaine-Mazur-style conjectures, where we consider $p$-adic representations that are finitely and tamely ramified: they are respectively known as {\it Tame Fontaine-Mazur Conjecture} (TFMC) \cite[5a]{FM} and {\it Tame Fontaine-Mazur Conjecture - Uniform version} (or {\it Uniform Tame Fontaine-Mazur Conjecture}, UTFMC).

\begin{conjecture}[Tame Fontaine-Mazur Conjecture] \label{TFMC}  Let $K$ be a number field and $S$ be a finite set of places of $K$ that are prime to $p$. Let $G_{S} := \Gal(K_{S}/K)$, where $K_{S}$ denotes the maximal pro-$p$-extension of $K$ that is unramified outside $S$. Then any $p$-adic representation of $G_{S}$ has finite image.
\end{conjecture} 

\noindent The idea behind this statement is that the eigenvalues of a Frobenius element must become roots of unity under the action of a finitely tamely ramified $p$-adic representation. In this case, the image of such a representation is solvable, hence finite by class field theory.  Class field theory also helps to prove that the conjecture holds for one-dimensional representations: we do this in Section \ref{CFT-TFMC}. For higher-dimensional representations, Conjecture \ref{TFMC} seems for now out of reach in general. 

\begin{conjecture}[Tame Fontaine-Mazur Conjecture - Uniform version] \label{TFMC-Uniform}
Let $K$ be a number field, and $\Gamma$ be a uniform pro-$p$ group of dimension $d>2$ (hence infinite). Then there does not exist a finitely and tamely ramified Galois extension $L/K$ whose Galois group $\Gal(L/K)$ is isomorphic to $\Gamma$. 
\end{conjecture}
 
\noindent We elaborate on how these statements connect in Section \ref{UTFMC}, and on how the assumptions made matter. In particular, we discuss a counterexample to Fontaine-Mazur Conjecture when ramification at $p$ is allowed. In Section \ref{future plans}, we expose some connected questions that are of interest to us.\\

\section{General notations}\label{notation}
\noindent From now on, we fix a prime integer $p$. We let $\Q_{p}$ be the field of $p$-adic numbers and $\Z_{p}$ be its ring of integers. Since the ring $\Z_{p}$ is isomorphic to the projective limit $\lim\limits_{\substack{\longleftarrow \\ k \geq 1}} \Z/p^{k}\Z$, there exists, for any integer $k \geq 1$, a natural ring isomorphism $\psi_{k} : \Z_{p}/p^{k}\Z_{p} \simeq \Z/p^{k}\Z$. These ring isomorphisms allows us to define {\it principal congruence subgroups} in this $p$-adic setting as follows. Given any integer $n \geq 2$, we consider the special linear group $\SL_{n}(\Z_{p})$, which is a maximal open compact subgroup of $\SL_{n}(\Q_{p})$. For any integer $k \geq 1$, we can then define the $k^{th}$-principal congruence subgroup of $\SL_{n}(\Z_{p})$ as $\Gamma_{n,k} := \ker\left(\SL_{n}(\Z_{p}) \twoheadrightarrow \SL_{n}(\Z/p^{k}\Z)\right)$. Note that we use this $\Gamma_{n,k}$ notation instead of the classical $K_{n}(k)$ notation to make the connection with the Galois context more apparent in the sequel.\\

\noindent Assume that we are given a group $G$. For any subsets $X, Y$ of $G$ and any positive integer $n$, we write $X^{n}$ (resp. $[X,Y]$) for the subgroup of $G$ generated by $\{x^n | x \in X\}$ (resp. by $\{[x,y] := x^{-1}y^{-1}xy, \ (x,y) \in X \times Y\}$). As usual, for any subgroup $H$ of $G$, we let $[G : H]$ denote the index of $H$ in $G$.\\

\noindent Also, if $\Gamma$ is a topological space and if $\Omega$ is a subset of $\Gamma$, then $\overline{\Omega}$ denotes the closure of $\Omega$ in $\Gamma$. If $\Gamma$ is moreover a topological group, a closed subgroup of $\Gamma$ is called \textit{topologically characteristic} when it is stable under all {\bf continuous} group automorphisms of $\Gamma$. \\

\section{Preliminaries on uniform pro-$p$ groups}\label{uniform groups}
\noindent Our main reference for this section is \cite{DMSS}. The goal here is to recollect all the information we need on uniform pro-$p$-groups to understand how they connect to Fontaine-Mazur conjecture in the view of Boston's methods. The reader who is already familiar with uniform pro-$p$-groups can hence skip this section and directly go to Section \ref{BostonProof}. 

\subsection{Some basic definitions related to profinite groups} \hfill \\
We start by recalling some basic definitions, coming from \cite[Chapter 1]{DMSS}, which are fundamental for our purpose. The first definition we make is exactly \cite[Definition 1.1]{DMSS}.
\begin{definition}
\label{profinite} 
A \textit{profinite group} is a compact Hausdorff topological group whose open subgroups form a basis of neighbourhoods of the identity.
\end{definition}
\noindent According to \cite[Proposition 1.3]{DMSS}, this is equivalent to the usual definition of profinite groups as inverse limits of finite groups.

\begin{definition}\label{fin gen}
A profinite group $G$ is \textit{finitely generated} if it contains a finite subset $X$ such that $X$ topologically generates $G$, i.e. such that $G$ is the smallest closed subgroup of $G$ containing $X$.
\end{definition}

\noindent Given a profinite group, we can also define its Frattini subgroup, which is one of the key tools in the forthcoming study of uniform pro-$p$-groups, as follows (see \cite[Definition 1.8]{DMSS}).
\begin{definition}\label{Frattini subgroup}
Let $G$ be a profinite group. The {\it Frattini subgroup} is defined as $\Phi(G) := \bigcap M$, where the intersection runs over all maximal proper open subgroups of $G$.
\end{definition}
\subsection{Pro-$p$-groups: definition and basic properties} \hfill\\
Among all profinite groups, we will have a special interest in those coming as projective limits of $p$-groups: the so-called pro-$p$-groups, which we define now, following \cite[Definition 1.10]{DMSS}.

\begin{definition}\label{propgroup}
A \textit{pro-$p$ group} is a profinite group in which every open normal subgroup has index equal to a power of $p$.
\end{definition}

\noindent Note that \cite[Proposition 1.12]{DMSS} ensures that pro-$p$-groups are exactly groups that are isomorphic to inverse limits of $p$-groups. (Recall that a {\it $p$-group} is just a finite group of $p$-power order.)\\

\noindent A nice characterisation of finitely generated pro-$p$-groups comes from the topology of their Frattini subgroup, as proven in \cite[Proposition 1.14]{DMSS}.
\begin{proposition}\label{key fact 1}
A pro-$p$-group $G$ is finitely generated if, and only if, its Frattini subgroup $\Phi(G)$ is open in $G$.
\end{proposition}

\noindent In the case of pro-$p$-groups, we can interestingly go beyond Frattini subgroups via the notion of lower $p$-series, defined as follows in \cite[Definition 1.15]{DMSS}.
\begin{definition}
\label{lower p-series}
The {\it lower $p$-series} of a pro-$p$-group $G$ is the series $(P_{i}(G))_{i \geq 1}$ of topologically characteristic subgroups of $G$ defined by $P_{1}(G) = G$ and: $\forall \ i \geq 1, \ P_{i+1}(G) := \overline{P_{i}(G)^{p}[P_{i}(G), G]}$.
\end{definition}

\noindent The next key proposition shows how the Frattini subgroup of a pro-$p$-group is actually encoded in the lower $p$-series \cite[Proposition 1.13]{DMSS}.
\begin{proposition}\label{key fact 2}
For any pro-$p$-group $G$, we have $\Phi(G) = P_{2}(G) = \overline{G^{p}[G,G]}$.
\end{proposition}

\noindent If $G$ is a finitely generated pro-$p$-group, then the quotient $G/\Phi(G)$ is a finite-dimensional $\F_{p}$-vector space (see \cite[Proposition 3.7]{DMSS} for further details on this statement). We can hence define the following integer.
\begin{definition}\label{dimension of fg pro-p group} 
For any finitely generated pro-$p$-group $G$, we let $d(G)$ be the dimension of $G/\Phi(G)$ as $\F_{p}$-vector space: 
\[ \displaystyle d(G) := \dim_{\F_p} G/\Phi(G) \ . \] 
\end{definition}
\noindent According to \cite[Proposition 3.7]{DMSS}, this integer is also equal to the minimal cardinality of a topological generating set for $G$. We will see later that, in good cases, $d(G)$ is also equal to the dimension of $G$ as a $p$-adic analytic group, provided that such a structure exists on $G$.

\subsection{The (uniform) power of pro-$p$-groups} \hfill\\
We still consider the case of (pro-)$p$-groups, and we introduce now the following notion of powerfulness (as in \cite[Definition 2.1]{DMSS}).

\begin{definition}\label{powerful}
We say that a $p$-group $G$ is {\it powerful} when one of the following holds:
\begin{itemize}
\item either $p$ is odd, and $G/G^{p}$ is an abelian group; 
\item or $p = 2$, and $G/G^{4}$ is an abelian group.
\end{itemize}
We say that a pro-$p$-group $G$ is {\it powerful} when the analoguous alternative, with $G^{p}$ and $G^{4}$ respectively replaced by their closure $\overline{G^{p}}$ or $\overline{G^{4}}$ in $G$, holds.
\end{definition}
\noindent According to \cite[Corollary 3.3]{DMSS}, powerful pro-$p$-groups can equivalently be defined as inverse limits of powerful $p$-groups in which all transition maps are surjective.\\

\noindent We are now ready to define the first of our key notions, namely \textit{uniform pro-$p$-groups}, as follows.\cite[Definition 4.1]{DMSS}.
\begin{definition}\label{uniform}
A \textit{uniformly powerful pro-$p$ group} $G$ (or: \textit{uniform pro-$p$-group}) is a powerful, finitely generated pro-$p$-group $G$ such that: $\forall \ i \geq 1, \ [P_i(G) \colon P_{i+1}(G)] = [G \colon P_2(G)]$.
\end{definition}

\noindent Note that the assumptions in Definition \ref{uniform} ensure that $P_{2}(G) = \Phi(G)$ is open in $G$. Another exciting feature of uniform pro-$p$-groups is that they are automatically endowed with a structure of $p$-adic analytic group. The reader interested in more information about the general setting of $p$-adic analytic groups should refer to Lazard's seminal paper \cite{L}, which remains to our knowledge the best reference so far on this topic. For now, we only need the following existence result \cite[Theorem 8.32]{DMSS}. 
\begin{theorem}\label{p-adic-structure}
A topological group has a structure of $p$-adic analytic group if, and only if, it contains an open subgroup which is a uniform pro-$p$-group.
\end{theorem}
\noindent In particular, this theorem ensures that uniform pro-$p$-groups themselves always have a structure of $p$-adic analytic group. According to \cite[Theorem 8.36]{DMSS}, their dimension as analytic group is then equal to the integer $d(G)$ given by Definition \ref{dimension of fg pro-p group}.

\subsection{A key example: uniform pro-$p$ groups and principal congruence subgroups} \hfill \\
\noindent The goal of this subsection is to prove the following result \cite[Theorem 5.2]{DMSS}, where the principal congruence subgroups $\Gamma_{n,k}$ are those we defined in Section \ref{notation}.
\begin{theorem}
\label{ThmMainExampleSLn}
Let $n \geq 2$ be an integer.
\begin{enumerate}
\item If $p$ is odd, then the first principal congruence subgroup $\Gamma_{n,1} \subset \SL_{n}(\Z_{p})$ is a uniform pro-$p$-group of dimension $n^{2} - 1$.
\item If $p = 2$, then the second principal congruence subgroup $\Gamma_{n,2} \subset \SL_{n}(\Z_{2})$ is a uniform pro-$p$-group of dimension $n^{2} - 1$.
\end{enumerate}
\end{theorem}
\noindent For instance, we obtain for $n = 2$ that for $p \geq 3$, $\Gamma_{1} := \ker\left(\SL_{2}(\Z_{p}) \twoheadrightarrow \SL_{2}(\F_{p})\right)$ is a uniform pro-$p$-group of dimension $3$, while $\Gamma_{2} := \ker\left(\SL_{2}(\Z_{2}) \twoheadrightarrow \SL_{2}(\Z/4\Z)\right)$ is a pro-$2$-group of dimension $3$.

\noindent Also note that, for $n = 1$, $\Gamma_{1,m}$ is the trivial group for any integer $m \geq 1$, and any prime $p$, which explains why we assume $n \geq 2$ in the previous theorem. \\

\noindent Though the proof of this fact is essentially given in \cite[Theorem 5.2]{DMSS}, we collect here the relevant facts and definitions on which it is based to ease later reference. For convenience, once $n$ is fixed, we write $\Gamma_{m}$ instead of ${\Gamma}_{n,m}$ to avoid this cumbersome notation as often as possible. Otherly said, we now fix an integer $n \geq 2$ and a prime $p$ and we set:
\[ \forall \ m \geq 1, \ \Gamma_{m} := \ker(\SL_{n}(\Z_{p}) \to \SL_{n}(\Z/p^{m}\Z)) \ . \]

\begin{lemma}\label{pro-p group}
For every $m \geq 1$, the group $\Gamma_{m}$ is a pro-$p$ group. In fact, we have
\[ \displaystyle \Gamma_{m} \cong \varprojlim_{r \geq 1} \Gamma_{m}/\Gamma_{m+r},
\]
with each $\Gamma_{m}/\Gamma_{m + r}$ being a finite $p$-group. Moreover, $\Gamma_{m}/\Gamma_{m+1}$ is isomorphic to the additive group $\Sl_{n}(\F_{p})$ of matrices in $M_{n}(\F_{p})$ with trace 0.
\end{lemma}

\begin{proof}
Let $m$ be a positive integer. The obvious projection map $\displaystyle \Gamma_{m} \to \varprojlim_{r \geq 1} \Gamma_{m}/\Gamma_{m+1}$ is injective since we have $\displaystyle \bigcap_{r \geq 1} \Gamma_{m+r} = \{I_n\}$. Conversely, any compatible sequence in $\displaystyle \varprojlim_{r \geq 1} \Gamma_{m} /\Gamma_{m+r}$ provides an element of $M_{n}(\Z_p)$ that is congruent to $I_{n}$ modulo $p^{m}$. By continuity of the determinant function, this element must have determinant $1$, so we have the desired isomorphism.

We now show that each $\Gamma_{m}/\Gamma_{m+r}$ is a (finite) $p$-group.  By induction on $r$, it is enough to prove that $\Gamma_{m}/\Gamma_{m+1}$ is a finite $p$-group. To do this, we consider the map $\varphi_{m} : \Gamma_{m} \to M_{n}(\Z_{p})$ defined by $\varphi_{m}(x) := p^{-m}\left(x - I_{n}\right)$. One directly checks that $\varphi_{m}(\Gamma_{m+1})$ is contained in $M_{n}(p\Z_{p})$, so we have a well-defined function $\overline{\varphi}_{m} : \Gamma_{m}/\Gamma_{m+1} \to M_{n}(\Z_{p})/M_{n}(p\Z_{p}) \simeq M_{n}(\F_{p})$.

We claim that $\overline{\varphi_{m}}$ is a group homomorphism, for the usual additive group structure on $M_{n}(\F_{p})$. Indeed, let $I_{n} + p^{m}a$ and $I_{n} + p^{m}b$ be elements of $\Gamma_{m}$: as $(I_{n} + p^{m}a)(I_{n} + p^{m}b) = I_{n} + p^{m}(a+b) + p^{2m}ab$, we have that 
\[\displaystyle \varphi_{m}(I_{n} + p^{m}a)(I_{n} + p^{m}b)) = \varphi_{m}(I_{n} + p^{m}(a+b) + p^{2m}ab) = a+b + p^{m}ab \ . \]

\noindent Since $a + b + p^{m}ab \equiv a + b \bmod p$, we obtain as expected that 
\[ \displaystyle \overline{\varphi_{m}}\left(I_{n} + p^{m}a)(I_{n} + p^{m}b)\right) = a + b \bmod p = \overline{\varphi_{m}}\left(I_{n} + p^{m}a)\right) + \overline{\varphi_{m}}\left(I_{n} + p^{m}b)\right) \ . \]

\noindent Now assume that $[I_{n} + p^{m}a] \in \Gamma_{m}/\Gamma_{m+1}$ is in $\ker \overline{\varphi_{m}}$. This means that $a$ belongs to $M_{n}(p\Z_{p})$, which implies that $I_{n} +.p^{m}a = I_{n}$ belongs to $\Gamma_{m+1}$, hence that $[I_{n} + p^{m}a] \in \Gamma_{m}/\Gamma_{m+1}$ is trivial and that $\overline{\varphi}_{m}$ is injective. Since $M_{n}(\F_{p})$ is a finite-dimensional $\F_{p}$-vector space, we obtain that $\Gamma_{m}/\Gamma_{m+1} \simeq \mathrm{Im}\ \overline{\varphi_{m}}$ is a finite $p$-group.

\noindent To conclude, we are left to check that the image of $\overline{\varphi_{m}}$ is isomorphic to $\Sl_{n}(\F_{p})$. First, consider an element $I_{n} + p^{m}a$ of $\Gamma_{m}$ and let $f_{a}(X) = \displaystyle \sum_{k = 0}^{n} \alpha_{k}X^{k} \in \Z_{p}[X]$ be the characteristic polynomial of $a$. Then the characteristic polynomial of $-p^{m}a$ is equal to $\displaystyle f_{-p^{m}a}(X) = \sum_{k = 0}^{n} \alpha_{k}(-p^{m})^{n-k}X^{k}$. So we have in particular that $1 = \det(I_{n} + p^{m}a) = f_{-p^{m}a}(1) \ = \displaystyle \sum_{k = 0}^{n} \alpha_{n-k} (-p^{m})^{k}$. Reducing this equality mod $p^{m}$ shows that $\alpha_{0} - 1 \in p^{m}\Z_{p}$, hence subtracting $1$ to it and dividing the result by $p^{m}$ gives that 
\[\displaystyle - \alpha_{n-1 } + \sum_{k = 2}^{n} \alpha_{n-k}(-p^{m})^{k-1}  = 0 \ .\]
This proves in particular that $- \alpha_{n-1} = \tr(a)$ is in $p^{m}\Z_{p}$, hence that $\tr(a) \equiv 0 \bmod p$, which proves that $\overline{\varphi_{m}}\left(I_{n} + p^{m}a\right)$ lies in $\Sl_{n}(\F_{p})$ as required.\\
Conversely, let us check that the image of $\overline{\varphi_{m}}$ is equal to $\Sl_{n}(\F_{p})$. First recall that, since $\overline{\varphi_{m}}$ is injective, it is also a homomorphism of $\F_{p}$-vector spaces. We are hence left to prove that a basis of $\Sl_{n}(\F_{p})$ is contained in $\mathrm{Im} \ \overline{\varphi_{m}}$ to conclude.
For $1 \leq i, j \leq n$, let $e_{ij}$ be the matrix whose $(i,j)$-th entry is $1$ and all other entries are $0$.  For $1 \leq i \leq n-1$, set $d_i := e_{ii} - e_{nn}$.  Then an $\F_{p}$-basis for $\Sl_n(\F_p)$ is given by
\[\displaystyle \left\{e_{ij} \colon 1 \leq i \neq j \leq n\right\} \sqcup \left\{d_i \colon 1 \leq i \leq n-1\right\} \ . \]
As $\varphi_{m}(I_{n} + p^{m}e_{ij}) = e_{ij}$ for all $1 \leq i \neq j \leq n$ while $\varphi_{m}(I_{n} + p^{m}d_{i}) = d_{i}$ for all $1 \leq i \leq n-1$, we obtain that $\varphi_{m}(\Gamma_{m}/\Gamma_{m+1}) = \Sl_{n}(\F_{p})$, as claimed.
\end{proof}

Until the end of this section, we essentially follow \cite[Section 5.1]{DMSS}, making the necessary changes to pass from $\GL_{n}$ to $\SL_{n}$ and filling in some missing details. First note that, if $x \in \Gamma_{m}$ with $m \geq 1$, then $x^{p} \in \Gamma_{m+1}$. Indeed, we can write $x = I_{n} + p^{m}a$ for some $a \in M_{n}(\Z_{p})$, and as $a$ and $I_{n}$ commute, we have
\[
\displaystyle x^{p} = (I_{n} + p^{m}a)^{p} = I_{n} + \sum_{k = 1}^{p} {p \choose k}p^{mk}a^{m} \equiv I_{n} \bmod p^{m+1} \ ,
\] 
since $p$ divides ${p \choose k}$ for all $1 \leq k \leq p-1$. Following \cite[Lemma 5.1]{DMSS}, we will now see that, unless $p = 2$ and $m = 1$, the $p$-th power map $\Gamma_{m} \to \Gamma_{m+1}$ is actually surjective.

\begin{lemma}\label{p power surjective}
Let $p$ be a prime and $m$ be a positive integer.
If $p$ is odd, or if $p = 2$ and $m \geq 2$, then every element of $\Gamma_{m+1}$ is the $p$-th power of an element of $\Gamma_{m}$.
\end{lemma}
\begin{proof}
Let $a \in M_{n}(\Z_{p})$ be such that $\det(I_{n} + p^{m+1}a) = 1$. We are looking for some $x \in M_{n}(\Z_{p})$ such that $\det(I_{n} + p^{m}x) = 1$ and $(I_{n} + p^{m}x)^{p} = I_{n} + p^{m+1}a$. We will find it by successive approximations, which means that we will produce a sequence $(x_{k})_{k \geq 0}$ of elements of $M_{n}(\Z_{p})$ such that:
\[\displaystyle \forall \ k \geq 0, \ \det(I_{n} + p^{m}x_{k}) \equiv 1 \bmod p^{m+1+k} \ \text{ and } \ (I_{n} + p^{m} x_{k})^{p}  \equiv I_{n} + p^{m}a \bmod p^{m+1+k} \ .\]

Then $\displaystyle x := \lim_{r \to \infty} x_r$ will satisfy the expected conditions. First note that the congruence condition $I_{n} + p^{m}a \equiv (I_{n} + p^{m-1}x_{r})^p \bmod p^{m+r}$ together with the equality $\det(I_{n} + p^{m}a) = 1$ already ensures that $\det(I_{n} + p^{m-1}x_{r})^p \equiv 1 \bmod p^{m+r}$. Since we know that $\det(I_{n} + p^{m-1}a) \in 1 + p\Z_{p}$, it follows that $\det(I_{n} + p^{m-1}x_{r}) \equiv 1 \bmod p^{m+r}$. This shows that we do not need to check whether the determinant congruence holds when constructing the sequence $(x_{r})_{r \geq 1}$ as it will automatically be true.

We start by constructing $x_{1}$.  Note that, so long as $m$ is in the supposed range, we have
\[ (I_{n} + p^{m-1}a)^p = I_{n} + p^{m}a + \sum_{k = 2}^p {p \choose k}(p^{m-1}a)^k \equiv I_{n} + p^{m}a \bmod p^{m+1} \ . \]
Thus we may set $x_{1} := a$. 

Now suppose that there exists some integer $r \geq 1$ for which we built some $x_{r} \in M_{n}(\Z_{p})$ satisfying $I_{n} + p^{m}a \equiv (I_{n} + p^{m-1}x_{r})^p \bmod p^{m+r}$. Then there exists some $c \in M_{n}(\Z_{p})$ such that
\[
(I_{n} + p^{m-1}x_{r})^{p} = I_{n} + p^{m}a + p^{m+r}c \ .
\]
Now note that we could expand the left-hand side of this equality as a $\Z$-linear combination of powers of $x_{r}$:  this shows that both $a$ and $c$ can be expressed as a $\Q_p$-linear combination of powers of $x_{r}$, which implies that all of $a, c$ and $x_{r}$ commute. Let us set
\[
z :=  (I_{n} + p^{m-1}x_{r})^{-(p-1)}c \text{ and } x_{r+1} := x_{r} - p^{r}z \ .
\]
\noindent Then $x_{r+1}$ satisfies the desired congruence, as we have:
\begin{align*}
(I_{n} + p^{m-1}x_{r+1})^{p} &= ((I_{n} + p^{m-1}x_{r}) - p^{m+r-1}z)^{p}\\
&= \sum_{k = 0}^p {p \choose k}(I_{n} + p^{m-1}x_{r})^{p-k}(-p^{m+r-1}z)^{k}\\
&= (I_{n} + p^{m-1}x_{r})^{p} + \sum_{k = 1}^p(-1)^{k}{p \choose k}(I_{n} + p^{m-1}x_{r})^{p-k}p^{(m+r-1)k}z^{k}\\
&= I_{n} + p^{m}a + p^{m+r}c + \sum_{k = 1}^p(-1)^{k}{p \choose k}(I_{n} + p^{m-1}x_{r})^{p-k}p^{(m+r-1)k}z^{k} \ ,
\end{align*}
which implies that
\begin{align*}
(I_{n} + p^{m-1}x_{r+1})^{p} &\equiv I_{n} + p^{m}a + p^{m+r}c - p(I_{n} + p^{m-1}x_{r})^{p-1}p^{m+r-1}z \bmod p^{m+r+1}\\
\text{ i.e. } (I_{n} + p^{m-1}x_{r+1})^{p} &\equiv I_{n} + p^{m}a \bmod p^{m+r +1}.
\end{align*}
As we have $x_{r+1} - x_{r} = -p^{r}z$ by construction, $\displaystyle x := \lim_{r \to \infty} x_r$ is well-defined, which ends the proof.
\end{proof}
The two next results come from \cite[Theorem 5.2]{DMSS}.
\begin{corollary}
\label{Pi of Gamma1} 
Let $p$ be a prime integer.
\begin{itemize}
\item If $p$ is odd, then $P_{i}(\Gamma_{1}) = \Gamma_{i}$ for all $i \geq 1$. 
\item If $p = 2$, then $P_{i}(\Gamma_{2}) = \Gamma_{i + 1}$ for all $i \geq 1$.
\end{itemize}
\end{corollary}
\begin{proof}
Set $\varepsilon := 0$ if $p$ is odd (resp. $1$ if $p = 2$). We show by induction on $i \geq 1$ that $P_{i}(\Gamma_{1 + \varepsilon}) = \Gamma_{i+\varepsilon}$. When $i = 1$, this follows from the definition of $P_{1}$, so we can suppose that $P_{i}(\Gamma_{1 + \varepsilon}) = \Gamma_{i + \varepsilon}$ for some $i \geq 1$.  Note that Lemma \ref{p power surjective} then ensures that $P_{i}(\Gamma_{1+\varepsilon})^{p} = \Gamma_{i+\varepsilon}^{p} = \Gamma_{i+1+\varepsilon}$.

We claim that $[P_{i}(\Gamma_{1+\varepsilon}), \Gamma_{1+\varepsilon}] = [\Gamma_{i+\varepsilon}, \Gamma_{1+\varepsilon}] \subseteq \Gamma_{i+1+\varepsilon}$. Indeed, let $I_{n} + p^{i+\varepsilon} \in \Gamma_{i+\varepsilon}$ and $I_{n} + p^{1+\varepsilon}b$ be two elements of $\Gamma_{1+\varepsilon}$.  Note that we have
\[
(I_{n} + p^{i+\varepsilon}a)^{-1} = I_{n} + \sum_{k = 1}^\infty (-1)^{k}p^{(i+\varepsilon)k}a^{k} \equiv I_{n} - p^{i+\varepsilon}a \bmod p^{i+1+\varepsilon}.
\]
\noindent This hence implies that we have
\begin{align*}
[I_{n} + p^{i+\varepsilon}a, I_{n} + p^{i+\varepsilon}b] &\equiv (I_{n} + p^{i+\varepsilon}a + pb)(1 - p^{i+\varepsilon}a)(I_{n} + pb)^{-1} \bmod p^{i+1+\varepsilon}\\
&\equiv (I_{n} + pb)(I_{n} + pb)^{-1} \bmod p^{i+1+\varepsilon}\\
&\equiv I_{n} \bmod p^{i+1+\varepsilon} \ ,
\end{align*}
\noindent hence shows that $[P_{i}(\Gamma_{1+\epsilon}), \Gamma_{1+\epsilon}] \subseteq \Gamma_{i+1+\epsilon}$, and thus we have $P_{i}(\Gamma_{1+\varepsilon})[P_{i}(\Gamma_{1+\varepsilon}), \Gamma_{1+\varepsilon}] \subseteq \Gamma_{i+1+\varepsilon}$.  Since $P_{i}(\Gamma_{1+\varepsilon}) = \Gamma_{i+1+\varepsilon}$, we must have equality. As $\Gamma_{i+1+\varepsilon}$ is open (hence closed) in $\Gamma_{1+\varepsilon}$, it follows that $P_{i+1}(\Gamma_{1+\varepsilon}) = \Gamma_{i+1+\varepsilon}$, as expected.
\end{proof}

\begin{theorem}
\label{first preliminary exercise}
Let $p$ be a prime integer.
\begin{itemize}
\item If $p$ is odd, then $\Gamma_{1}$ is a uniform pro-$p$-group of dimension $n^{2} - 1$. 
\item If $p = 2$, then $\Gamma_{2}$ is a uniform pro-$2$-group of dimension $n^{2}-1$.
\end{itemize}
\end{theorem}

\begin{proof}
As in the proof of Lemma \ref{p power surjective}, we will give a uniform proof for the two cases by setting $\varepsilon := 0$ if $p$ is odd (resp. $1$ if $p = 2$). According to Lemma \ref{pro-p group}, $\Gamma_{1+\varepsilon}$ is a pro-$p$ group. Propositions \ref{key fact 1} and \ref{key fact 2} hence ensure that $\Gamma_{1+\varepsilon}$ is finitely generated if, and only if, $P_{2}(\Gamma_{1+\varepsilon})$ is open in $\Gamma_{1+\varepsilon}$. But we know from Corollary \ref{Pi of Gamma1} that $P_{2}(\Gamma_{1+\varepsilon}) = \Gamma_{2+\varepsilon}$, which is open in $\Gamma_{1+\varepsilon}$, so the finite type condition is satisfied.

Now recall that Lemma \ref{p power surjective} ensures that $\Gamma_{1+\varepsilon}^{p} = \Gamma_{2+\varepsilon}$.  From Lemma \ref{pro-p group}, we get that $\Gamma_{1+\varepsilon}/\Gamma_{2+\varepsilon} \cong \Sl_{2}(\F_{p})$ is abelian, hence $\Gamma_{1+\varepsilon}$ is powerful.  Using similarly that Corollary \ref{Pi of Gamma1} gives that $P_{i}(\Gamma_{1+\varepsilon}) = \Gamma_{i+\varepsilon}$ while Lemma \ref{pro-p group} shows that $\Gamma_{j}/\Gamma_{j+1} \cong \Sl_{n}(\F_{p})$, we obtain that $P_{i}(\Gamma_{1+\varepsilon})$ also satisfies the necessary condition from Definition \ref{uniform} for all $i \geq 1$. All this finally shows that $\Gamma_{1+\epsilon}$ is uniform.

Finally, we calculate the dimension using Lemma \ref{pro-p group}, Corollary \ref{Pi of Gamma1} and Definition \ref{uniform} as follows: from these statements follows that we have $\Gamma_{1+\varepsilon}/\Phi(\Gamma_{1_\varepsilon}) = \Gamma_{1+\varepsilon}/P_{2}(\Gamma_{1+\varepsilon}) = \Gamma_{1+\varepsilon}/\Gamma_{2+\varepsilon} \cong \Sl_{n}(\F_{p})$.  Since $\displaystyle \dim_{\F_{p}} \Sl_{n}(\F_{p}) = n^{2}-1$, we are done.
\end{proof}

We end this section by a well-known result that will be used in Section \ref{UTFMC}. Its proof is as outlined in \cite[Page 1271]{M}.
Let $G$ be a finitely generated pro-$p$-group. For any $i \geq 1$, set $H^{i}(G) : = H^i(G, \F_p)$ and let $d_{p}(G) := \dim_{\F_p}H^{1}(G)$ be the $p$-rank of $G$.
\begin{theorem}
\label{Zp-quotient} 
Any uniform subgroup of dimension 1 or 2 has a quotient isomorphic to $\Z_p$.  
\end{theorem}
Proving this result requires the following result of Lazard \cite[V, Proposition (2.5.7.1)]{L}.
\begin{theorem} 
 \label{Lazardadmis}
Let $G$ be a uniform group of positive dimension $d$. Then, for all $i \geq 1$, one has 
\[ \displaystyle H^{i}(G) \cong \bigwedge^{i}(H^{1}(G)) \ , \] 
where the exterior product is induced by the cup product. 
\end{theorem}
\begin{proof}[Proof of Theorem \ref{Zp-quotient}] Suppose that $G$ is a uniform pro-$p$-group of dimension $d \geq 1$. If $d =1$, then $G$ is clearly isomorphic to $\Z_{p}$ so we are done. If $\dim G = \dim_{\F_{p}}H^{1}(G, \F_{p}) = 2$, then Theorem \ref{Lazardadmis} implies that $\dim_{\F_{p}}H^{2}(G, \F_{p}) = \dim_{\F_p}\bigwedge^{2}H^{1}(G, \F_{p}))  = 1$, i.e. that $H^{2}(G, \F_{p})$ is actually isomorphic to $\F_{p}$. Now note that $H^{1}(G, \F_{p}) = H^{1}(G^{ab}, \F_{p})$. Since $G^{ab}$ is a finitely generated  $\Z_{p}$-module, we can write it as
\begin{equation}
\label{DecompoGab}
\displaystyle G^{ab} \simeq \Z_{p}^{r} \times \prod_{s=1}^{n} \Z/p^{i_{s}}\Z  
\end{equation}
for some nonnegative integers $r, n, i_{1}, \ldots, i_{n}$. Recall that the short exact sequence
\[ \displaystyle 1 \longrightarrow \F_{p} \longrightarrow \Q_{p}/\Z_{p} \stackrel{\times p }{\longrightarrow } \Q_{p}/\Z_{p} \longrightarrow 1 \]
leads to the following long exact sequence of cohomology: 
\[ 0 \longrightarrow H^{1}(G^{ab}, \F_{p}) \longrightarrow H^{1}(G^{ab}, \Q_{p}/\Z_{p}) \stackrel{\times p }{\longrightarrow} H^{1}(G^{ab}, \Q_{p}/\Z_{p}) \longrightarrow H^{2}(G, \F_{p}) \longrightarrow \ldots \] 
Let $K$ and $C$ respectively denote the kernel and the cokernel of $\times p$, and let $d_{p}(K)$ and $d_{p}(C)$ denote their respective dimensions as $\F_{p}$-vector spaces. From what we said above directly follows that
\[ \displaystyle d_{p}(C) \leq d_{p}(H^{2}(G, \F_{p}))  \  \text{ and }  \ d_{p}(K)  \geq d_{p}(H^{1}(G, \F_{p})) \ , \]
hence we have $d_{p}(K) - d_{p}(C) \geq  d_{p}(H^{1}(G, \F_{p})) - d_{p}(H^{2}(G, \F_{p})) = 1$. Since \eqref{DecompoGab} also ensures that $d_{p}(K) - d_{p}(C) = r = \mathrm{rk}_{\Z_{p}}G^{ab}$, we obtain that $\mathrm{rk}_{\Z_{p}}(G^{ab}) \geq 1$, hence $G^{ab}$ surjects onto $\Z_{p}$.  
\end{proof}


\section{Boston's proof of a special case of Fontaine-Mazur Conjecture}
\label{BostonProof}
In this section, we study the proof of Theorem \ref{Boston-main} as given by Boston in \cite{B1, B2}. It heavily relies on Lazard's extensive study of $p$-adic analytic Lie groups, as given in \cite{L}. Recall in particular that Lazard defined the notion of $p$-saturated groups and used it to give an algebraic characterisation of $p$-adic analytic groups, as topological groups containing a topologically finitely generated, open, $p$-saturated pro-$p$-group with an integer-valued filtration. In \cite{DMSS}, the notion of uniform pro-$p$-groups is used to transpose Lazard's work in a group-theoretic maneer and to obtain a analogue characterization of $p$-adic analytic groups for odd $p$,where uniform pro-$p$-groups play the same role as $p$-saturated pro-$p$-groups with integer-valued filtration do in Lazard's statement \cite[Page 81]{DMSS}. Note that in \cite{B1}, Boston used Lazard's terminology ({\it $p$-saturated with integer values}) instead of the {\it uniform group} one. Let us recall here the main statement of \cite{B1} (namely Theorem 1), as originally stated.

\begin{theorem}
\label{RecallBostonMain} Let $K$ be a normal extension of prime degree $\ell \neq p$ of a number field $F$ such that $p \nmid h(F)$, the class number of $F$. Then there is no infinite everywhere unramified Galois pro-$p$ extension $L$ of $K$ such that $L$ is Galois over $F$ and $\Gal(L/K)$ is $p$-saturated with integer values. 
\end{theorem}
This result gives some evidence for Conjecture \ref{FMC3} in a special case. Using the strength of fixed point free automorphisms, Boston generalised this result to the case of cyclic extensions $L/K$ where $[L : K]$ is not necessarily a prime, see \cite[Theorem 1]{B2}. To do this, he introduced a class of `self-similar' groups, which contains the class of `uniform' groups, and showed that, under some condition $H(G,n)$ related to fixed-point free automorphisms (that is conjectured to always hold), Theorem \ref{RecallBostonMain} carries over for cyclic finite extensions of degree co-prime to $p$, with `uniform' replaced by `self-similar'. As it appeared above, stating this more general result requires to define some new notions.

\begin{definition} 
\label{selfsimilargroups} A pro-$p$-group $G$ is said {\it self-similar} when it has a filtration by open, characteristic subgroups $G = G_{1} \subseteq G_{2} \subseteq \cdots $ such that $G_{i}/G_{i+1}$ is abelian for all $i \geq 1$ and $\displaystyle \bigcap G_{i} = \{ 1\}$, together with a family of group isomorphisms 
\[ \displaystyle \phi_{i}: G_{i}/G_{i+1} \to G_{i+1}/ G_{i+2} \]
that commute with every continuous automorphism of $G$.
\end{definition}

\noindent Note that the commutativity of all quotients $G_{i}/G_{i+1}$ ensures that we only have to check the second condition for outer automorphisms of $G$. Also note that Definition \ref{selfsimilargroups} implies that self-similar groups are always infinite, since $\displaystyle \lim_{i \to \infty} |G/ G_{i}| = \lim_{i \to \infty}| G/G_{2}|^{i-1} = \infty$.

The next proposition justifies our interest in this notion.
\begin{proposition}
\label{uniformareselfsimilar}
Any uniform group is a self-similar group.
\end{proposition}
\begin{proof}
Let $G$ be a uniform pro-$p$-group. For any $i \geq 1$, set $G_{i} := P_{i}$, where $P_{i}$ is as in Definition \ref{lower p-series}. Then $(G_{i})_{i \geq 1}$ is a filtration as in Definition \ref{selfsimilargroups} and the map $[x \mapsto x^{p}]$ induces group isomorphisms $\varphi_{i} : G_{i-1}/G_{i} \to G_{i}/G_{i+1}$ that commute with any continuous automorphism of $G$, hence $G$ is indeed a self-similar group.
\end{proof}

We now make explicit the aforementioned condition $H(G,n)$ defined by Boston in \cite[Definition 2]{B2}.
\begin{definition}
\label{H(G, n)} 
Let $G$ be a pro-$p$ group and $n$ be a positive integer. We say that {\it $H(G, n)$ holds} when there is a function of $n$ that is an upper bound for the derived length of every finite quotient of $G$ that admits a fixed point free automorphism of order $n$. 
\end{definition}
We already mentioned above that $H(G,n)$ is conjectured to hold for any pro-$p$-group $G$ and any integer $n \geq 1$. So far, it is known to hold when:
\begin{itemize}
\item $n$ is either a prime integer or $n = 4$, and $G$ arbitrary;
\item $G$ is such that the rank of all its finite quotients is bounded, and $n$ arbitrary.
\end{itemize}
The latter case covers the uniform groups case, see \cite[p. 52]{DMSS}, hence for $G$ being a uniform group, there is a (uniform) bound on the derived length of the quotients $G/P_{i}$ for $i \geq 1$.

We can now state the generalisation of Theorem \ref{RecallBostonMain} as proven in \cite[Theorem 1]{B2}.
\begin{theorem}[Boston]
\label{B2-main}
Let $F$ be a number field such that $p$ does not divide the class number of $F$. Let $K$ be a cyclic extension of $F$ of degree $n \geq 2$ co-prime to $p$. Then there is no infinite, everywhere unramified, Galois pro-$p$-extension $L$ of $K$ such that $L$ is Galois over $F$  and ${\rm Gal}(L/K)$ is self-similar and satisfies the property $H({\rm Gal}(L/K), n)$ as stated in Definition \ref{H(G, n)}. 
\end{theorem} 
Before proving this theorem, we recall a classical result of Schur and Zassenhaus \cite[Chapter 4]{RZ}. 
\begin{theorem}[Schur-Zassenhaus] 
Let $1 \to \Gamma \to \mathcal{G} \to \mathcal{G}/\Gamma \to 1$ be a short exact sequence of profinite groups, with $\Gamma$ a finitely generated pro-$p$-group and $\mathcal{G}/\Gamma$ of finite order co-prime to $p$. Then $\mathcal{G}$ contains a subgroup $\Delta_{0}$ isomorphic to the quotient $\Delta := \mathcal{G}/\Gamma$, and $\Delta_{0}$ is unique up to conjugation in $\mathcal{G}$. In particular, $\mathcal{G}$ is isomorphic to a semi-direct product of $\Delta$ and $\Gamma$: $\mathcal{G} = \Gamma \rtimes \Delta_0 \cong \Gamma \rtimes \Delta$. 
\end{theorem}

\begin{proof}[Proof of Theorem \ref{B2-main}] To make things more confortable to the reader, we only deal here with uniform groups, but the general case of self-similar groups can be proven along the same lines, as done in \cite{B2}. Suppose by contradiction that there exists an infinite, everywhere unramified, Galois pro-$p$-extension $L$ of $K$ such that $L/F$ is Galois and $G := {\rm Gal}(L/K)$ is uniform and satisfies condition $H(G,n)$, where $n := [L:K]$. By Schur-Zassenhaus theorem, the following extension splits:
\[ \displaystyle 1 \longrightarrow G \to {\rm Gal}(L/F) \longrightarrow {\rm Gal}(K/F) \longrightarrow 1 \ . \]
Let us pick an element $\sigma$ of ${\rm Gal}(L/F)$ that lifts a generator of the cyclic group ${\rm Gal}(K/F)$ under the splitting above. As $G$ is a normal subgroup of ${\rm Gal}(L/F)$, $\sigma$ induces an action by conjugation on $G$. Now, since $G$ is a uniform group, we can consider its filtration by its characteristic subgroups $P_{i}$ (as given in Definition \ref{lower p-series}). Each of the $P_{i}$ is preserved under the action by conjugation of $\sigma$ (as it is a characteristic subgroup of $G$), hence we get a conjugation action of $\sigma$ on each of the quotients $G/P_{i}$.

Suppose that, for all $i \geq 1$, $\sigma$ has no non-trivial fixed point in $G/P_{i}$. As $G$ satisfies $H(G, n)$, there is a uniform bound on the derived length of each of the quotients $G/P_{i}$ as $i$ goes to $\infty$. Such a bound cannot exists since a repeated use of the finiteness of class number shows that the maximal unramified pro-$p$-extension of any fixed derived length must be a finite extension.

This shows that there exists a positive integer $i$ such that $\sigma$ does not act fixed-point free on $G/P_{i}$. Assume that $i$ is minimal for this property and let $\tau$ be a non-trivial fixed point of $\sigma$ in $G/P_{i}$: $\sigma\cdot \tau = \tau$. By minimality of $i$, $\tau$ must maps to the identity element in $G/P_{i-1}$, i.e. lands in $P_{i-2}/P_{i-1}$, since we have compatibility of the action of $\sigma$ with $ P_{i-1} \supseteq P_{i}$ and $G/P_{i} \twoheadrightarrow G/P_{i-1}$. Now recall that Proposition \ref{uniformareselfsimilar} ensures that the isomorphism $\phi_{i-1} : P_{i-2}/P_{i-1} \to P_{i-1} /P_{i+2}$ is $\sigma$-equivariant, hence $\varphi_{i-1}^{-1}(\tau)$ should define a non-trivial fixed point for $\sigma$ in $P_{i-2}$, which contradicts the minimality of $i$ if $i \geq 3$.

If $i = 2$, then our assumption is that the action by conjugation of $\sigma$ on $G/P_{2}$ has a non-trivial fixed point. By definition, we have $P_{2} = \overline{G^{p}[G, G]}$, hence $G/P_{2}$ is naturally an $\F_{p}$-vector space. Otherly said, it defines a representation over a field of characteristic $p$ of the group $\langle \sigma \rangle$ generated by $\sigma$. As this group is, by definition of $\sigma$, of cardinality $\mid \langle \sigma \rangle \mid = \mid \mathrm{Gal}(K/F) \mid = n$, which is prime to $p$, we can apply Maschke's theorem to decompose $G/P_{2}$ as a direct sum of irreducible representations of $\langle \sigma \rangle$ over $\F_{p}$. Since $\tau$ is fixed under the action of $\sigma$, this decomposition can be written as $G/P_{2} = \F_{p}\tau \oplus W$ for some representation $W$ of $\langle \sigma \rangle$ over $\F_{p}$.

The direct sum decomposition ensures that $W$ also defines a subgroup and a quotient of $G/P_{2}$, so we can define an extension $M$ of $K$ such that $\mathrm{Gal}(M/K) = \langle \tau \rangle$ by letting $M$ be the fixed field (in $L$) of $W$ over $K$. The field extension $M/K$ is then an abelian extension of prime degree $p$, and we have a tower of fields of the form $F \subset K \subset M$.

We now claim that there exists a cyclic extension of degree $p$ over $F$ that is contained in $M$. Indeed, recall that $\mathrm{Gal}(K/F) \simeq \langle \sigma \rangle$ acts on $\mathrm{Gal}(M/K)$ by conjugation. Since $K/F$ is assumed to be a normal extension, $\mathrm{Gal}(M/K)$ must be a normal subgroup of $\mathrm{Gal}(M/F)$. We can hence apply Schur-Zassenhaus theorem to write $\mathrm{Gal}(M/F)$ as a semi-direct product of the following form:
\begin{equation}
\label{SDProduct}
\displaystyle \mathrm{Gal}(M/F)\ \simeq \ \mathrm{Gal}(M/K) \rtimes \mathrm{Gal}(K/F) \ .
\end{equation}
As $\tau$ is fixed under $\sigma$, the action of $\sigma$ on $\mathrm{Gal}(M/K) = \langle \tau \rangle$ is trivial, hence $\mathrm{Gal}(K/F)$ is also normal in $\mathrm{Gal}(M/F)$, which turns \eqref{SDProduct} into a direct product of the two groups appearing on the right hand side. As $\mathrm{Gal}(M/K)$ is of order $p$ while $\mathrm{Gal}(K/F)$ is of order $n$, with $n$ and $p$ co-prime, we can conclude that there exists a Galois extension $N$ of $F$ of degree $p$ (hence cyclic) in $M$.
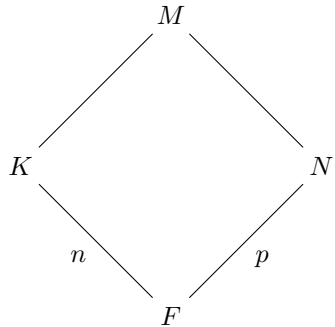
\begin{center}
\begin{tikzpicture}[node distance = 2cm, auto]
      \node (F) {$F$};
      \node (K) [above of=F, left of=F] {$K$};
      \node (N) [above of=F, right of=F] {$N$};
      \node (M) [above of=F, node distance = 4cm] {$M$};
      \draw[-] (F) to node {$n$} (K);
      \draw[-] (F) to node [swap] {$p$} (N);
      \draw[-] (K) to node {} (M);
      \draw[-] (N) to node {} (M);
      \end{tikzpicture}
\end{center}  
Since $M/F$ is unramified, so is $N/F$, which is in contradiction with the fact that $p$ is supposed to be co-prime to $h(F)$. In all cases ($i \geq 3$ or $i = 2$), we have a contradiction, which proves that there exists no pro-$p$-extension $L/K$ that is everywhere unramified and has for Galois group a uniform group.
\end{proof}


\section{Some results on Tame Fontaine-Mazur conjecture and its uniform version}
\label{UTFMC} 
For this section, we will mostly refer to \cite{H, HM, RR}. The goal of this section is to present some results attached to the Tame Fontaine-Mazur Conjecture (TFMC for short) and the Tame Fontaine-Mazur Conjecture-Uniform version (TFMC-U for short), as well as related questions that some of the authors are currently working on.
 
\subsection{Motivation and background}
We start by recalling some results that are at the core of this section, namely the two versions of Fontaine-Mazur Conjectures mentioned above that will be proven below to be equivalent to each other.

As explained in Section \ref{intro}, TFMC is an analogue of Conjecture \ref{FMC3} when considering $p$-adic representations that are finitely and tamely ramified. Since a theorem of Grothendieck \cite[Appendix]{ST} ensures that tame representations are automatically potentially semi-stable, Fontaine-Mazur conjecture must imply (under some standard conjectures in algebraic geometry, see \cite{KW} for more details) that the following statement holds \cite[Conjecture 5a]{FM}.
\vspace{0.5\baselineskip}

Fix a prime $p$. For a number field $K$ and a (finite) set of places $S$ of $K$ all of which are co-prime to $p$, let $K_{S}$ be the maximal pro-$p$-extension of $K$ that is unramified outside of $S$.  Write $G_{S}:= \Gal(K_{S}/K)$.

\begin{conjecture}[Tame Fontaine-Mazur Conjecture] Let $p$ be a prime integer and $K$ be a number field. Let $S$ be a finite set of places of $K$ that are all co-prime to $p$ and let $K_{S}$ be the maximal pro-$p$-extension of $K$ that is unramified outside $S$. Set $G_{S} := \mathrm{Gal}(K_{S}/K)$. Then, for any positive integer $n$,  any continuous Galois representation $\rho: G_{S} \to GL_{n}(\Q_{p})$ has finite image.
\end{conjecture} 
In the next section (Section \ref{CFT-TFMC}), we will see that when $n =1$, TFMC follows by class field theory. For $n > 1$, this conjecture appears in general completely out of reach, though some preliminary evidence exists, as those given by \cite{B2, H, W}. When $K = \Q$ and $n =2$, we pointed out in Section \ref{intro} that the main contribution so far is due to Kisin by different methods \cite{K}.

Recall that Theorem \ref{p-adic-structure} claims that any finitely generated $p$-adic analytic group contains a uniform open subgroup. Moreover, Theorem \ref{Zp-quotient} asserts that any uniform group of dimension $1$ or $2$ admits a quotient isomorphic to $\Z_{p}$. We can hence rephrase TFMC as follows.

\begin{conjecture}[Tame Fontaine-Mazur Conjecture - Uniform version for $(K,d)$] Let $K$ be a number field and $\Gamma$ be a uniform pro-$p$-group of dimension $d > 2$ (hence infinite). Then there is no finitely and tamely ramified Galois extension of $K$ whose Galois group is isomorphic to $\Gamma$.
\end{conjecture}

In the light of Conjectures \ref{TFMC} and \ref{TFMC-Uniform} for $n = 2$, we would like to underline the major contribution due to Hajir and Maire in \cite{HM} before going further. Among several results of importance, they proved an analogue of Theorem \ref{Boston-main} for odd $p$ by the mean of purely group-theoretic  and arithmetic methods such as the effect of a semi-simple cyclic action with fixed points on the group structure, the rigidity of uniform groups, the existence of Minkowski units, some arithmetic properties of Galois groups, etc.

To make this more precise, we need to introduce some notation and definitions, following \cite[1.2]{HM}. With the notation above, let $T$ be an auxiliary finite set of places of $K$ {such that $T \cap S = \emptyset$. Define $K^{T}_{S}$ as the maximal pro-$p$-extension of $K$ that is unramified outside $S$ and in which any place in $T$ splits completely, and let $G_{S}^{T} := \mathrm{Gal}(K^{T}_{S}/K)$ be the corresponding Galois group. Note that $K_{S}^{T}$ is a subfield of $K_{S}$ while $G^{T}_{S}$ is a quotient of $G_{S}$, and that $K_{S}^{\emptyset} = K_{S}$.

We can now state one of the main results of Hajir-Maire \cite[Section 1.2, Theorem]{HM}, which can be seen as a special case of more general statements proven in \cite[Section 2]{HM}.
\begin{theorem}[Hajir-Maire]
Let $K/k$ be a quadratic extension with Galois group $\langle \sigma \rangle$. Assume that the odd prime $p$ does not divide the class number of $k$. Suppose that for any finite set $\Sigma$ of places of $k$ all co-prime to $p$, there is no continuous Galois representation $G_{\Sigma}(k) \twoheadrightarrow \Gamma_{2,1}$. Then there exist infinitely many disjoint finite sets $S$ and $T$ of primes of $K$ such that any place of $S$ is prime to $p$, $|S|$ is arbitrarily large and:
\begin{enumerate}
\item  $G^{T}_{S}$ is finite;
\item $G^{T}_{S}/\Phi(G^{T}_{S})$ has $\vert S \vert$ independent fixed points under the action of $\sigma$;
\item there is no continuous representation $\rho: G^{T}_{S} \twoheadrightarrow \Gamma_{2,1}$ whose image is a uniform group.  
\end{enumerate}
\end{theorem}
Note that in the last assertion, the uniformness assumption is actually strengthened into a notion of $\sigma$-uniformness that gives further condition on the action of $\sigma$ on $G^{T}_{S}$, see \cite[Definition 1.1]{HM}. Let us also mention that, though it is not so obvious at first sight, this result is actually tightly connected to Fontaine-Mazur Conjecture for uniform groups of constant type for order $2$ automorphisms. The interested reader can check \cite[Corollary 2.6]{HM} for more details on this topic.

\subsection{Class field theory implies Tame Fontaine-Mazur Conjecture for $n=1$}
\label{CFT-TFMC} \hfill \\
\noindent The goal of this section is to prove the tame Fontaine-Mazur Conjecture (Conjecture \ref{TFMC}) for $n=1$. This may be known to experts in the field, but we think this is important to have a reference where this claim is actually proven, hence we now give a full proof of the following statement.
\begin{theorem}[TFMC for $n = 1$] 
\label{ThmTFMCn=1}
Let $K$ be a number field and $p$ be a rational prime. Let $S$ be a finite set of place of $K$ in which no place has residue characteristic $p$. Let $K_{S}$ denote the maximal pro-$p$-extension of $K$ that is unramified outside $S$ and $G_{S} := \mathrm{Gal}(K_{S}/K)$ be its Galois group. Then every continuous group homomorphism 
\[\displaystyle \chi : G_{S} \to \GL_{1}(\Q_{p}) = \Q_{p}^{\times}\]
has finite image.
\end{theorem}
We keep the notation above throughout the rest of this section and prove (using class field theory) that $\chi$ must have finite image. First note that, since $G_{S}$ is a compact group (as it is a pro-$p$-group), is image under the continuous map $\chi$ must be a compact subgroup of $\Q_{p}^{\times}$, hence lands into $\Z_{p}^{\times}$. Also note that, since $\Q_{p}^\times$ is an abelian group, $\chi$ must factor through the abelianisation $G_{S}^{\ab}$ of $G_{S}$, which motivates the connection with class field theory. The latter provides indeed an injective group homomorphism with dense image (called Artin reciprocity law)
\[
\rho : \A_{K}^{\times}/K^{\times} \hookrightarrow \Gal(\overline{K}/K)^{\ab} \ ,
\]
that becomes an isomorphism when $\A_{K}^{\times}/K^{\times}$ is replaced by its profinite completion. We are hence reduced to show that the composite group homomorphism $\tilde{\chi} := \chi \circ \rho : \A_{K}^{\times}/K^{\times} \to \Z_{p}^{\times}$ has finite image.  

Saying that $\chi$ is unramified outside $S$ means that $\chi$ is trivial on inertia groups for places outside $S$, hence so is $\tilde{\chi}$. Now let $v$ be a place in $S$. We need to distinguish between the finite and archimedean cases as follows. 

If $v$ is a finite place, then the inertia group at $v$ corresponds to the subgroup of $\A_{K}^\times$ that is trivial at all places outside of $v$ and equal at $v$ to the subgroup $\OK_{K_{v}}^\times$ of $K_{v}^\times$. As all places in $S$ are co-prime to $p$, the group of $1$-units in $\OK_{K_{v}}^{\times}$ is a pro-$\ell$-group for some rational prime $\ell \not= p$. Since any continuous group homomorphism from a pro-$\ell$-group into a pro-$p$-group is necessarily trivial, knowing that the $1$-units of $\Z_{p}^{\times}$ is a pro-$p$-group ensures that the image by $\tilde{\chi}$ of $\OK_{K_{v}}^{\times}$ is finite. Letting $S^{\infty}$ denote the subset of finite places in $S$, we hence have proven that $\tilde{\chi}\left(\prod_{v \in S^{\infty}} \OK_{K_{v}}^{\times} \right)$ is finite.

We now study what happens on $\displaystyle \prod_{v {\text{ finite place}}} \varpi_{v}^{\Z}$, where $\varpi_{v}$ denotes a uniformising element of $\OK_{v}$. First note that the product is actually a restricted product, which means only finitely many places are such that $\varpi_{v}$ has non-trivial image in $\Z_{p}^{\times}$. Also recall that any element $x = (\varpi_v^{e_{v}})_{v \text{ finite }}$ of $\displaystyle \prod_{v {\text{ finite place}}} \varpi_{v}^{\Z}$ defines a fractional ideal $\Aa_{x}$ of $K$, namely $\displaystyle \prod_{v} \p_{v}^{e_{v}}$, where $\p_{v}$ denotes the prime ideal of $\OK_{K}$ corresponding to $v$. (This is well-defined as only finitely many $e_{v}$ can be non-zero.) If $h$ denotes the class number of $K$, then $\Aa_{x}^{h}$ is a non-zero principal ideal of $\OK_{K}$ by definition of $h$. This implies that $x \mod K^{\times}$ must have finite order dividing $h$. Thus we obtain that the quotient of $\displaystyle \prod_{v {\text{ finite place}}} \varpi_{v}^{\Z}$ by the diagonal copy of $K^{\times}$ (in $\A_{K}^{\times}/K^{\times}$) has finite exponent dividing $h$, hence its image by $\tilde{\chi}$ must satisfy the same condition in $\Q_{p}^{\times}$. As $\Q_{p}^{\times}$ only contains finitely many elements of order dividing $h$, we can conclude that this image by $\tilde{\chi}$ is finite, as expected.

Finally assume that $v$ is an archimedean place in $S$. As $\Z_{p}^{\times}$ is a totally disconnected space, the identity component of $K_{v}^{\times}$ (that is either $\R_{>0}$ if $v$ is real, or $\C^{\times}$ is $v$ is complex) has trivial image under the continuous group homomorphism $\tilde{\chi}$. Letting $S_{\infty}$ denote the set of archimedean places in $S$, we are hence left to see that $\displaystyle \prod_{v \in S_{\infty} \text{ real }} \R^{\times}/\R_{>0}$ has finite image under $\tilde{\chi}$, which is straightforward as $S$ and $\R^{\times}/\R_{>0} \simeq \Z/2\Z$ are finite sets, so the proof is complete.

\subsection{Equivalence of TFMC and TFMC-U}\hfill \\
This subsection aims to prove the equivalence of Conjectures \ref{TFMC} (TFMC) and \ref{TFMC-Uniform} (TFMC-U). First note that we have the following nice result coming from \cite[Theorem 4.5]{DMSS}, which makes unnecessary the assumption that $G$ is infinite in TFMC-U.
\begin{proposition}
\label{uniform implies torsion free}
Any uniform pro-$p$-group is torsion-free.
\end{proposition}
 
\begin{proof}[{\bf TFMC $\implies$ TFMC-U}] 
Let $G$ be a uniform pro-$p$ group, and suppose that there exists a finitely and tamely ramified Galois extension $L/K$ with $\Gal(L/K) \cong G$. To prove that $G$ is necessarily finite, we will use for $G$ the following key fact, proven in \cite[Section 7.3]{DMSS}.
\begin{proposition}
Every pro-$p$-group of finite rank admits a faithful representation into $\GL_{n}(\Q_{p})$ for some integer $n \geq 1$.
\end{proposition} 
Hence, there is an integer $n \geq 1$ for which there exists a faithful representation $\iota \colon G \hookrightarrow \GL_{n}(\Q_{p})$. Fixing an isomorphism $\psi$ between $\mathrm{Gal}(L/K)$ and $G$ as given in the statement, the composition of $\iota$ by $\psi$ and by the natural projection map $\Gal(\overline{K}/K) \to G$ provides a Galois representation
\[ \displaystyle \rho \colon \Gal(\overline{K}/K) \to \GL_{n}(\Q_{p}) \ . \]
Denote by $S$ the set of places in $K$ that ramify in $L$: by assumption, $S$ is a finite set and $\rho$ factors through $G_{S}$. We are now left to prove that any place in $S$ is prime to $p$, as TFMC would hence be applicable to $\rho$ and give that $\mathrm{Im}(\rho) \simeq G$ is finite, as expected.

Recall that $L/K$ is assumed to be tamely ramified. This ensures that for any prime $\p$ of $K$ and any prime $\q$ of $L$ above $\p$, the extension $L_{\q}/K_{\p}$ is tamely ramified, and there exists a unique intermediate Galois extension $K_{\p} \hookrightarrow F \hookrightarrow L_{q}$ such that $L_{q}/F$ is totally ramified and $F/K_{p}$ is unramified. The `tame quotient' of the decomposition group of $\q$ above $\p$ is then defined by $\Gal(L_{q}/F)$, and it is in particular of (pro-)order prime-to-$p$. Since $G$ is a pro-$p$-group, this implies that $L/K$ must be unramified at places above $p$, hence that any place in $S$ is prime to $p$, and we are done.
 \end{proof}
 
\begin{proof}[{\bf TFMC-U $\implies$ TFMC}] Conversely, assume that $K$ is a number field and that $S$ is a finite set of places of $K$ that are all prime to $p$, and let $\rho : G_{S} \to \GL_{n}(\Q_{p})$ be a continuous Galois representation. To use TFMC-U, we need to get a uniform pro-$p$-group our of this setting. First note that since $G_{S}$ is compact while every maximal compact subgroup of $GL_{n}(\Q_{p})$ is conjugate to $\GL_{n}(\Z_{p})$, we can assume without loss of generality that $\rho$ takes values in $\GL_{n}(\Z_{p})$.

Set $G := \mathrm{Im}(\rho) \subset \GL_{n}(\Z_{p})$. We claim that $G$ is a $p$-adic analytic group. Indeed, \cite[Corollary 8.34]{DMSS} ensures that it is enough to prove that $G$ contains an open subgroup $H$ that is a pro-$p$-group of finite rank. Let us set
\[
\mathcal{G} := \ker(\GL_n(\Z_p) \to \GL_n(\Z/p\Z)) \ \text{ and } \ H := G \cap \mathcal{G} \ . 
\]
Then $H$ is certainly a pro-$p$-group (as $G$ is), and is open in $G$ (as $\mathcal{G}$ is open in $\GL_{n}(\Z_{p})$), so we are left to check that $H$ is of finite rank. Recall that the rank of $H$ is defined as 
\[\displaystyle \mathrm{rk}(H) :=  \sup_{J \leq_c H} d(J) \ ,
\] 
where $d(J)$ denotes the minimal number of elements needed to topologically generate $J$, for $J$ a closed subgroup of $H$ (compare with \cite[Proposition 3.11 and Definition 3.12]{DMSS}). 

Note that $H$ is also closed in $\GL_{n}(\Z_{p})$, as intersection of two closed subgroups of $\GL_{n}(\Z_{p})$, namely $G$ (that is the image of a compact set under a continuous map) and $\mathcal{G}$ (that is open in $\GL_{n}(\Z_{p})$, hence closed). As a consequence, any closed subgroup of $H$ is also closed in $\mathcal{G}$. Now recall that $\mathcal{G}$ is a uniform pro-$p$-group of dimension $d(\mathcal{G}) = \rk(\mathcal{G}) = n^{2}$ (see \cite[Theorem 5.2]{DMSS}, which follows the same path as the proof of Theorem \ref{ThmMainExampleSLn} given above). This implies, thanks to \cite[Theorem 3.8]{DMSS}, that any closed subgroup $J$ of $H$ (hence of $\mathcal{G}$ too) satisfies
\[
d(J) \leq d(\mathcal{G}) = n^{2}.
\]
In particular, this shows that $H$ has finite rank, hence that $G$ is a $p$-adic analytic group. Applying \cite[Corollary 8.34]{DMSS} then shows that $G$ contains an open normal uniform pro-$p$-subgroup $\Gamma$ of finite index. Denote by $F$ the subfield extension of $K_{S}^{\ker\rho}/K$ on which $\Gamma$ acts trivially. As $\Gamma$ is normal of finite index in $G = \mathrm{Im}(\rho)$, $F/K$ is a finite Galois extension. We can now apply TFMC-U to $F/K$ and $\Gamma$ to obtain that $\Gamma$ has to be finite (and actually trivial because of Proposition \ref{uniform implies torsion free}), hence so does $G = \mathrm{Im}(\rho)$. Note that the ramification conditions of TFMC-U are satisfied since $\rho$ is unramified outside $S$ while the places in $S$ are all prime to $p$, so we are done.
\end{proof}
 
\subsection{The small dimensional cases of TFMC ($d = 1$ and $d = 2$)} \hfill \\
In this section, we will briefly see why we do not have to consider the cases $d \leq 2$ in Conjecture \ref{TFMC-Uniform}. First recall that any uniform pro-$p$-group of dimension $1$ is isomorphic to $\Z_{p}$, hence can be represented into $\GL_{1}(\Z_{p}) = \Z_{p}^{\times}$ using the map $p \mapsto 1+p$. This reduces us to the case $n=1$ of TFMC, which was proven from class field theory in Section \ref{CFT-TFMC} (see Theorem \ref{ThmTFMCn=1} in particular).

Now assume that $\Gamma$ is a uniform pro-$p$-group of dimension $2$. Using \cite[Exercise 3.11]{DMSS}, we know that $\Gamma$ is {\it meta-procyclic}, which means that there exists an element $\gamma$ in $\Gamma$ such that $\overline{\langle \gamma \rangle}$ is a normal subgroup of $\Gamma$ and the quotient $\Gamma/\overline{\langle \gamma \rangle}$ is pro-cyclic. Suppose that there exists a finitely and tamely ramified Galois extension $L/K$ with $\mathrm{Gal}(L/K)$ isomorphic to $\Gamma$. Quotienting $\Gamma$ by $\overline{\langle \gamma \rangle}$ then provides a $\Gamma/\overline{\langle \gamma \rangle} \simeq \Z_{p}$-extension of $K$ that is finitely and tamely ramified. This cannot happen as we just proved that TFMC-U holds for $d = 1$, so $L/K$ cannot exist and TFMC-U is proven for $d = 2$.

\subsection{Tame ramification and co-primality matters} 
\label{counter-example2} \hfill \\
The goal of this section is to discuss why the two main assumptions in TFMC (tame ramification and co-primality with $p$ of the places in $S$) cannot be dropped. We start by discussing the ramification assumption and give a simple example that shows why tame ramification is necessary in the statement of Conjecture \ref{TFMC}, i.e. we explain how to produce an interesting representation $\rho \colon G_{\Q, S} \to \GL_{2}(\Q_{p})$ that is wildly ramified at $p$ and whose image is infinite.

Let $\mathcal{E}$ be a non-CM elliptic curve defined over $\Q$ and $p$ be any prime integer. By Serre's open image theorem \cite{Serre72}, we know that the $p$-adic representation of $G_{\Q}$ attached to $\mathcal{E}$ has open image in $\GL_{2}(\Z_{p})$ (and is actually surjective for almost all $p$). Also note that this representation is unramified away from $p$ and from the primes dividing the conductor of $\mathcal{E}$, which certainly provides a finite set $S$. We are hence left to see that this representation is wildly ramified at $p$ to conclude. To do so, we notice that the determinant of this representation is the $p$-adic cyclotomic character $\chi_{p} \colon G_{\Q} \to \Z_{p}^\times$. As the tame quotient of $G_{\Q_{p}}$ is prime to $p$ while $\Z_{p}^{\times}$ is (up to a finite part) a pro-$p$-group, we get that $\chi_{p}$ is wildly ramified at $p$, hence so is the $p$-adic representation of $G_{\Q}$ attached to $\mathcal{E}$.

One can now wonder whether one can get rid of this wild ramification by simply twisting away the determinant. To deal with this, we need to set up a bit more notation first. We fix a prime integer $p >2$ and let $\rho \colon G_{\Q} \to \GL_{2}(\Z_{p})$ be the $p$-adic representation attached to $\mathcal{E}$: we hence have $\det \rho = \chi_{p}$. Let $\langle \ \cdot \ \rangle \colon \Z_{p}^\times \to 1 + p\Z_{p}$ denote the projection of $\Z_{p}^{\times}$ onto its pro-$p$ part. As $p$ is odd, every element of $1 + p\Z_{p}$ has a unique squareroot that is congruent to $1$ mod $p$, so we get a character $\varepsilon : G_{\Q} \to 1 + p\Z_{p}$ whose square equals $\langle \chi_p \rangle$. If we set $\rho' := \rho \otimes \varepsilon^{-1}$, then $\det \rho'$ is of finite order, so we cannot use the previous argument to show that $\rho'$ is wildly ramified at $p$. Nevertheless, note that the surjectivity of $\rho$ would ensure that $\mathrm{Im}(\rho')$ contains $\SL_2(\Z_p)$, and otherwise $\mathrm{Im}(\rho')$ will at least contain an open subgroup of $\SL_{2}(\Z_{p})$, so $\mathrm{Im}(\rho')$ cannot be finite, and TFMC predicts that $\rho'$ must be wildly ramified at $p$. We can prove this last assertion without using TFMC:  indeed, first observe that if the image of the inertia contains an open subgroup of a unipotent subgroup, i.e. the conjugate of an open subgroup of 
\[
U = \left\{\begin{pmatrix}
1 & u\\
0 & 1 
\end{pmatrix} \colon u \in \Z_{p}\right\} \ ,
\]
then so does $\mathrm{Im}(\rho')$ as $\det(U) = \{1\}$. As $U$ is a pro-$p$-group while the tame inertia quotient has pro-order prime-to-$p$, $\rho'$ cannot be tamely ramified at $p$, hence must be wildly ramified at $p$, as expected.

\begin{remark} There are many one-dimensional counterexamples: for instance, any integer $k$ defines a one-dimensional representation $\chi_{p}^{(p-1)k}$ that is unramified outside $p$ and wildly ramified at $p$. Nevertheless, such representations do not arise from geometry (since they are not De Rham representations), which is not very satisfying in the context of Fontaine-Mazur conjectures, the latter being motivated by understanding representations coming from geometry.
\end{remark}

Let us now focus on the other main assumption, namely that the places in $S$ must be co-prime to $p$. First recall that any $\Z_{p}$-extension must be unramified at primes with residue characteristic different from $p$. As $\Z_{p}$ is abelian, the finiteness of the class number requires that $\Z_{p}$-extensions cannot be everywhere unramified, hence there must be some ramification at places above $p$. In particular, this implies that the requirement claiming that $S$ contains no place above $p$ is necessary in Conjecture \ref{TFMC}.

Following Fontaine and Mazur \cite[page 44]{FM}, one can wonder whether the finiteness of $S$ is automatic, or if it is an actual assumption, at least for semisimple finite-dimensional $p$-adic representations. Ramakrishna answered this question when $n = 2$, proving in \cite{R} that under GRH, there exists an explicit irreducible $2$-dimensional $p$-adic representation that ramifies at infinitely many primes but is potentially semistable at $p$. Let us moreover mention that in \cite{KRR}, Khare and Ramakrishna proved that the conditions  {\it $S$ is finite} and {\it $S$ does not contain places above $p$} required in the statement of Conjecture \ref{TFMC} are actually independent.

Since the finiteness of $S$ is not automatic even in small dimension, one can wonder whether it is a necessary condition for Conjecture \ref{TFMC} to hold. Again, this question has been addressed by Ramakrishna and his co-authors in \cite{KLR}. More precisely, recall that a $p$-adic representation is said {\it deeply ramified at a given prime} when it does not vanish on any of the corresponding higher ramification groups of finite index. With this definition in mind, one can rephrase our question as follows: can we find a $p$-adic representation that ramifies at infinitely many primes of a number field $K$ but is not deeply ramified at $p$? The answer is yes, at least when $n = 2$ and $p  > 5$, as proven by Khare-Larsen-Ramakrishna in \cite[Main Theorem]{KLR}.

\section{Some future directions}\label{future plans} 
We end this paper by giving some future directions of research, some of them being work in progress, others being longer-term projects. First recall that the results of Hajir-Maire on which are based Section \ref{UTFMC}, but also the most recent results connected to these questions obtained by Maire \cite{M21}, require the prime $p$ to be odd. A first natural question we hence want to explore in the future is the following one.
\begin{question}
To which extend do the previous results remain valid for $p = 2$?
\end{question}

Another interesting question, standing at the core of an ongoing project of two of the authors \cite{AP21}, is to study an explicit example of biquadratic field pointed out by Boston and Buell \cite[page 290]{B1} as being close to providing an actual counterexample to Conjecture \ref{FMC3}. More precisely, we are interested in the following question.
\begin{question}
Set $K =\Q (\sqrt{-104}, \sqrt{229})$, $p= 3$ and let $L$ be the maximal everywhere unramified pro-$p$-extension of $K$.  Does $\Gal(L/K)$ have an infinite $p$-adic analytic pro-$p$-quotient?
\end{question}
Note that Boston's method cannot be used to test whether Conjecture \ref{FMC3} holds for this biquadratic field and this $p$, since we are not dealing with a cyclic extension. This hence leads to wonder whether we could extend Boston's method beyond the cyclic case, i.e. when $\Gal(K/\Q)$ is a non-cyclic extension.
\subsubsection*{Acknowledgments} We want to thank the organizers of the workshop {\it Women in Numbers Europe 3} for giving us this opportunity to work on this project. We also warmly thank Jaclyn Lang for many stimulating discussions and for earlier work on this project.


\end{document}